\documentclass{amsart}
\usepackage{amssymb}
\usepackage{stmaryrd} 
\usepackage{amsmath} 
\usepackage{amscd}
\usepackage{amsthm}
\usepackage{amsbsy}
\usepackage[margin=1.2 5in]{geometry}
\usepackage{commath, longtable}
\usepackage{comment, enumerate}
\usepackage{geometry}
\usepackage[matrix,arrow]{xy}
\usepackage{hyperref}
\usepackage{mathrsfs}
\usepackage{color}
\usepackage{float}
\usepackage{mathtools,caption}
\usepackage[table,dvipsnames]{xcolor}
\usepackage{tikz-cd}
\usepackage{longtable}
\usepackage[utf8]{inputenc}
\usepackage[OT2,T1]{fontenc}
\usepackage[normalem]{ulem}
\usepackage{hyperref}
\usepackage[customcolors]{hf-tikz}
\usepackage{enumitem}
\newlist{primenumerate}{enumerate}{1}
\setlist[primenumerate,1]{label={\arabic*$'$}}
\hypersetup{
 colorlinks=true,
 linkcolor=DarkOrchid,
 filecolor=blue,
 citecolor=olive,
 urlcolor=orange,
 pdftitle={Jha-Kundu-Majumdar project},
 }
\usepackage{booktabs}

\DeclareSymbolFont{cyrletters}{OT2}{wncyr}{m}{n}
\DeclareMathSymbol{\Sha}{\mathalpha}{cyrletters}{"58}

\newtheorem{theorem}{Theorem}[section]
\newtheorem{lemma}[theorem]{Lemma}

\newtheorem{proposition}[theorem]{Proposition}

\numberwithin{equation}{section}
\newtheorem{lthm}{Theorem}
\newtheorem{lpro}{Proposition}

\theoremstyle{remark}
\newtheorem{remark}[theorem]{Remark}
\newtheorem{example}[theorem]{Example}

\newcommand{\rk}{\operatorname{rk}}
\newcommand{\EC}{\mathsf{E}}

\newcommand{\Z}{\mathbb{Z}}
\newcommand{\Q}{\mathbb{Q}}

\makeatletter
\theoremstyle{plain} 
\newtheorem*{intr@thm}{\intr@thmname}

\newtheorem*{c@njecture}{\conjn@name}
\newcommand{\myl@bel}[2]{
 \protected@write \@auxout {}{\string \newlabel {#1}{{#2}{\thepage}{#2}{#1}{}} }
 \hypertarget{#1}{}
 } 

\makeatother

\makeatletter
\newcommand{\mylabel}[2]{#2\def\@currentlabel{#2}\label{#1}}
\makeatother

\title[]{Hilbert's $10^{\text{th}}$ Problem via Mordell curves}

\author[S.~Jha]{Somnath Jha}
\address[Jha]{Department of Mathematics and Statistics, IIT Kanpur, India}
\email{jhasom@iitk.ac.in}

\author[D.~Kundu]{Debanjana Kundu}
\address[Kundu]{Department of Mathematical and Statistical Sciences\\ UTRGV \\ 1201 W University Dr.\\ Edinburg, TX 78539\\ USA}
\email{dkundu@math.toronto.edu}

\author[D.~Majumdar]{Dipramit Majumdar}
\address[Majumdar]{Department of Mathematics, IIT Madras, India}
\email{dipramit@iitm.ac.in}


\begin{document}
\begin{abstract}
We show that for $5/6$-th  of all primes $p$, Hilbert’s $10^\text{th}$ problem is unsolvable for the ring of integers of $\Q(\zeta_3, \sqrt[3]{p})$. 
We also show that there is an infinite set $S$ of square-free integers  such that Hilbert’s $10^\text{th}$ problem is unsolvable over the ring of integers of $\Q(\zeta_3, \sqrt{D}, \sqrt[3]{p})$ for every $D \in S$ and for every prime $p \equiv 2, 5 \pmod 9$.
We use the CM elliptic curves $y^2=x^3-432 D^2$ associated to the cube-sum problem, with $D$ varying in suitable congruence class, in our proof.
\end{abstract}

\maketitle

\section*{Introduction}

In 1900, D.~Hilbert posed the following question (which was the tenth in his list of twenty-three questions): does there exist an algorithm (Turing machine) that takes as input polynomial equations over $\Z$ and decides whether they have integer solutions.
Building on the work of M.~Davis, H.~Putnam, and J.~Robinson in \cite{davis1961diophantine, davis1961decision, robinson1969unsolvable}, Y.~Matiyasevich proved that computably enumerable sets over $\Z$ are exactly the same as the Diophantine sets over $\Z$; see \cite{Mat70}.
In other words, he proved that Hilbert's $10^{\text{th}}$ problem is unsolvable.

Soon after in 1978, J.~Denef and L.~Lipshitz \cite{denef1978diophantine} asked whether there exists an algorithm (Turing machine) that takes as input polynomial equations over $\mathcal{O}_F$ (the ring of integers of a number field $F$) and decides whether they have integer solutions.
A straight-forward argument shows that if $\Z$ is Diophantine in $\mathcal{O}_F$, then the analogue of Hilbert’s $10^{\text{th}}$ problem for $\mathcal{O}_F$ has a negative solution.

At the start of the 21st century, it was first observed by B.~Poonen that there is a connection between rank stabilization  of elliptic curve in extensions of number fields and Hilbert's $10^{\text{th}}$ problem, \cite{poonen2002using}.
The result was further strengthened by A.~Shlapentokh (which is recorded in Theorem~\ref{Shlapentokh}) and has been one of the key tools in proving new cases of the conjecture of Denef--Lipshitz.
For several classes of number fields,  Denef--Lipshitz conjecture has been established. For example, (i) if a number field is totally real, (ii) if it a quadratic extension of a totally real number field, (iii) if the number field is abelian or (iv) if it has precisely one (pair of) complex places.
There are also important related works of B.~Mazur--K.~Rubin \cite{mazur2010ranks}.
For degree $6$ number fields of the form $\Q(\sqrt{-q}, \sqrt[3]{p})$ as $p$ and $q$ vary over  certain explicit set of primes of positive density,  Denef--Lipshitz conjecture was established in \cite{garcia2020towards, kundu2024studying}.
On the other hand, \cite{sw} proves that for a positive proportion of integers $a$, Hilbert's $10^{\text{th}}$ problem is unsolvable for $\mathcal{O}_F$ where $F$ is of the form $\Q(\sqrt[6]{a})$.

We consider Hilbert's $10^{\text{th}}$ problem for the ring of integers of certain $S_3$ extensions and degree 12 extensions of $\Q$. 
The Kummer extension $\Q(\zeta_3, \sqrt[3]{m}), m \in \Z$ is the first layer of the so called `false Tate curve' extension for $p=3$ and is widely studied in (non-commutative) Iwasawa theory.
Using Mordell curves $\EC_a: y^2 = x^3 + a, a \in \Z$ and its cubic twists, we apply results on the cube-sum problem to prove new cases of the Hilbert's $10^{\text{th}}$ problem. 
In particular, we prove the following results
\begin{lthm}
For $\frac{5}{6}$-th of all primes $p$, Hilbert's $10^{\text{th}}$ problem has a negative solution over the ring of integers of $\Q(\zeta_3, \sqrt[3]{p})$.
\end{lthm}
 
\begin{lpro}
Hilbert's $10^{\text{th}}$ problem is unsolvable over the ring of integers of $\Q(\zeta_3, \sqrt{D}, \sqrt[3]{p})$, where $D$  varies over an infinite set $S$ consisting of square-free integers and $p$ is a prime such that $p \equiv 2,5 \pmod 9$.
Further, $\#(S \cap [-X,X]) \gg X^{1 - \epsilon}$ where $X \gg 0$ and for any $\epsilon >0$.
\end{lpro}

We emphasize that this result complements the results proven in \cite{kundu2024studying} and indeed provides (infinitely many) new examples of number fields where Hilbert's $10^{\text{th}}$ problem is unsolvable.

We are interested in Hilbert's $10^{\text{th}}$ problem for the ring of integers of $\Q(\zeta_3, \sqrt[3]{m}), m\in \Z$ and it raised the following natural question which is of interdependent interest:
given a Mordell elliptic curve $\EC_a: y^2=x^3+a, a \in \Z$ does there exist a cube-free integer $D$ co-prime to $a$ such that precisely one of the elliptic curves $\EC_a, \EC_{aD^2}, \EC_{aD^4}$ has positive Mordell--Weil rank over $\Q$?
We did an extensive computation on this on SAGE/MAGMA which are mentioned in Section~\ref{Sec: SAGE computations}.
This led us to apply results from the cube-sum problem.

As we were completing our preprint, C.~Pagano--P.~Koymans announced an unconditional proof of the unsolvability of Hilbert's $10^{\text{th}}$ problem \cite{KP24} over the ring of integers of any number field.
Their result is accomplished by constructing elliptic curves $\EC$ and studying their quadratic twists to guarantee that there is no rank growth of $\EC$ in certain quadratic extensions. They use $2$-descent argument and tools from additive combinatorics. 
On the other hand, our approach involves studying cubic twists instead of quadratic twists of elliptic curves.
In fact, we specifically use  Mordell curves. 
We use results from the cube-sum problem to study the cubic twists of the Mordell curves. Thus our method is different.
Even more recently, L.~Alpöge--M.~Bhargava--W.~Ho--A.~Shnidman announced another proof \cite{ABHS}.

\subsubsection*{Organization}
Including this introduction, there are three sections in this short note.
Section~\ref{Sec: prelims} is preliminary in nature, where we record crucial results that will be used in the main argument.
In Section~\ref{Sec: SAGE computations}, we discuss the main motivation behind this note and also raise questions which we have not been able to answer completely.
In Section~\ref{Sec: main result}, we state and prove the main result of our note.

\section*{Acknowledgements}
We thank Alvaro Lozano-Robledo, Sudhanshu Shekhar, and Ari Shnidman for helpful discussions.
We thank Antonio Lei for his interest in our paper.
We thank the referee for their timely reading of the paper and for helpful comments; in particular, we thank them for suggesting using MAGMA to complete our list of examples.
SJ is supported by ANRF grant CRG/2022/005923.
DK is supported by an AMS-Simons Early Career Travel Grant.

\section{Preliminaries}
\label{Sec: prelims}

In this section, we record some results which will be used crucially in our argument(s).

\subsection{Cube-Sum Problem}

The classical Diophantine cube-sum problem  asks the question: which integers $D$ can be expressed as a sum of two rational cubes?
An integer $D$ is called a cube-sum (resp. not a cube-sum), if $x^3+y^3=D$ has (resp. does not have) a rational solution.
The elliptic curve $X^3+Y^3=DZ^3$ can be expressed in the  Weierstrass equation as $y^2=x^3-432D^2$ and for a cube-free integer $D>2$ it is known that torsion subgroup of ${\EC_{-432D^2}(\Q)}$ vanishes.
Thus, a cube-free integer $D > 2$ satisfies $D=a^3+b^3$ for $a, b \in \Q $ if and only if the Mordell-Weil group $ E_{-432D^2}(\Q)  $ has a positive rank.
The cube-sum problem has been studied in great detail with important contributions by J.~Sylvester, F.~Lucas, T.~P{\'e}pin, E.~Selmer, P.~Satg{\'e}, D.~Lieman, F.~R.~Villegas--D.~Zagier etc.
Some recent works include \cite{DV, cai2017cube, hu2019explicit, absbs}.
We refer to  \cite{DV,absbs} for a literature review.
We record the special cases below which will be used in the article.

\begin{theorem}
\label{Thm: JMS}
\leavevmode
\begin{enumerate}
\item[\textup{(}a\textup{)}]
Suppose that $p \equiv 2, 5 \pmod{9}$ is an odd prime. By \cite[Section 2]{Sylvester} the integers $p, p^2, 9p,$ and $9p^2$ are not cube-sum.
In other words for all primes $p \equiv 2,5 \pmod 9$, we have
\[
\rk_{\Z}\EC_{-432p^2}(\Q) = \rk_{\Z}\EC_{-432p^4}(\Q) = \rk_{\Z}\EC_{-432(9p)^2}(\Q) = \rk_{\Z}\EC_{-432(9p^2)^2}(\Q) = 0.
\]

\item[\textup{(}b\textup{)}]
Suppose that $p \equiv 4, 7 \pmod{9}$ is an odd prime.
\begin{enumerate}
\item[\textup{(}i\textup{)}] By \cite{DV}, if $3 \not\in \mathbb{F}_p^3$, then $p$ is a cube-sum.
In other words $\rk_{\Z}\EC_{-432p^2}(\Q) >0 $.

\item[\textup{(}ii\textup{)}] 
Suppose that $\ell \equiv 8\pmod{9}$ is a prime such that $\ell \notin \mathbb{F}^3_p$. By \cite{majumdar2023cube} the integers $\ell p, \ell p^2$, and $p \ell^2$ are not cube-sum.
Equivalently,
\[
\rk_{\Z}\EC_{-432(\ell p)^2}(\Q) = \rk_{\Z}\EC_{-432(\ell p^2)^2}(\Q) = \rk_{\Z}\EC_{-432(p \ell^2)^2}(\Q)= 0.
\]


\end{enumerate}

\item[\textup{(}c\textup{)}]
Let $d$ be a positive integer and $a \equiv 8 \pmod 9$ be an integer co-prime to $d$. By \cite[Theorem B]{JMS2}, there are infinitely many primes $p \equiv a \pmod{9d}$ such that $p$ is a cube-sum. 

\end{enumerate}    
\end{theorem}

\subsection{Hilbert's \texorpdfstring{$10^{\text{th}}$}{} Problem and Elliptic Curves}

\begin{theorem}[A.~Shlapentokh {\cite{shlapentokh2008elliptic}}]
\label{Shlapentokh}
Let $L/F$ be an extension of number fields.
Suppose that there is an elliptic curve $\EC$ defined over $F$ such that 
\[
\rk_{\Z} \EC(L) = \rk_{\Z} \EC(F) >0
\]
Then $\mathcal{O}_F$ is Diophantine in $\mathcal{O}_L$.

If further, it is known that $\Z$ is Diophantine in $\mathcal{O}_F$ (i.e., the analogue of Hilbert’s $10^{\text{th}}$ problem for $\mathcal{O}_F$ has a negative solution) then (by transitivity) $\Z$ is Diophantine in $\mathcal{O}_L$ and the analogue of Hilbert’s $10^{\text{th}}$ problem for $\mathcal{O}_L$ has a negative solution.    
\end{theorem}

A useful corollary of Shlapentokh's theorem was proved in \cite[Proposition~3.3]{garcia2020towards}, which we record below.

\begin{theorem}
\label{GFP prop 3.3}
Let $F/\Q$ and $K/\Q$ be extensions of number fields with $K/\Q$ quadratic.
Let $L$ be the compositum of $F$ and $K$ over $\Q$.
Suppose that there is an elliptic curve $\EC/\Q$ satisfying the conditions:
\begin{enumerate}
    \item[\textup{(}i\textup{)}] $\rk_{\Z}\EC(F) = 0$ \textit{ and }
    \item[\textup{(}ii\textup{)}] $\rk_{\Z}\EC(K) > 0$.
\end{enumerate}
Then $\mathcal{O}_L/\mathcal{O}_F$ is integrally Diophantine.  

If further, it is known that the analogue of Hilbert’s $10^{\text{th}}$ problem for $\mathcal{O}_F$ has a negative solution, then the analogue of Hilbert’s $10^{\text{th}}$ problem for $\mathcal{O}_L$ has a negative solution.    

\end{theorem}

\section{Some Computations}
\label{Sec: SAGE computations}

\subsection{Sufficient conditions}

Let $a$ be a non-zero integer and consider the elliptic curve
\[
\EC = \EC_a: y^2 = x^3 + a
\]
It is well known that $\EC_a$ is an elliptic curve with complex multiplication by $\Z[\zeta_3]$ and there is a degree $3$ rational isogeny $\phi_a:\EC_a \rightarrow \EC_{-27a}$. It follows that $\rk_{\Z}\EC(K) = 2\rk_{\Z}\EC(\Q)$, here $K$ denotes the number field $K=\Q(\zeta_3)$.
Fix a cube-free (rational) integer $D$ and define the elliptic curves
\begin{align*}
    \EC_1 = \EC_{aD^2} : y^2 & = x^3 + D^2 a\\
    \EC_2 = \EC_{aD^4} : y^2 & = x^3 + D^4 a.
\end{align*}
The curves $\EC_1$ and $\EC_2$ are the cubic twists of $\EC$, i.e., the three elliptic curves $\EC$, $\EC_1$, and $\EC_2$ are isomorphic over $F = \Q(\sqrt[3]{D})$.
In particular, the three elliptic curves are isomorphic over the Galois closure of $F$, in other words over the extension $L = F\cdot K = \Q(\zeta_3, \sqrt[3]{D})$.

Recall that
\begin{equation}\label{eq1}
\rk_{\Z}\EC(L) =\rk_{\Z}\EC_1(L) = \rk_{\Z}\EC_2(L)  = \rk_{\Z}\EC(K) + \rk_{\Z}\EC_1(K) + \rk_{\Z}\EC_2(K)
\end{equation}

Now, suppose that for a fixed cube-free integer $D$, there exists an elliptic curve
\[
\EC: y^2 = x^3 + a_D
\]
such that \textit{exactly one} of $\EC, \ \EC_1, \ \EC_2$ has positive Mordell--Weil rank $r$ over $\Q$ and the other two have Mordell--Weil rank 0.
Without loss of generality, suppose that $\rk_{\Z}\EC(\Q)=r>0$.
Then $\rk_{\Z}\EC(K) =2r$ and by the assumption in the previous paragraph, it follows that $\rk_{\Z}\EC_1(K) = \rk_{\Z}\EC_2(K) =0$.
This implies
\[
\rk_{\Z}\EC(L) = \rk_{\Z}\EC(K) = 2r + 0 + 0 =2r >0.
\]
Since Hilbert's $10^{\text{th}}$ Problem is unsolvable in $K=\Q(\zeta_3)$, using Shlapentokh's result (Theorem~\ref{Shlapentokh}) we deduce the following lemma:
\begin{lemma}\label{lem1}
If there exists  an elliptic curve $\EC: y^2 = x^3 + a$ such that upon twisting by a cube-free integer $D$ exactly one of $\EC/\Q, \ \EC_1/\Q, \ \EC_2/\Q$ has positive Mordell--Weil rank, then Hilbert's $10^{\text{th}}$ problem is unsolvable for $\mathcal{O}_L$ where $L=\Q(\zeta_3, \sqrt[3]{D})$.
\end{lemma}

\subsection{Data}

We now present the data for the following SAGE computations.
We fixed an elliptic curve $\EC: y^2 = x^3 + a_D$ and varied $1 < D \le 100$ over cube-free integers to check in which cases exactly one of the cubic twists had positive rank over $\Q$.
By the earlier discussion we know that Hilbert's $10^{\text{th}}$ problem is unsolvable in $\mathcal{O}_L$.

\begin{example}
Let $D$ be an integer in the set
\begin{equation*}
\begin{split}
\{3, 5, 7, 9, 10, 11, 13, 21, 23, 25, 26, 28, 29, 31, 33, 39, 41, 44, 45, 46, 47, 49, 51, 55, 59, 61, 65, 66, 67, \\
69, 75, 76, 77, 82, 87, 91, 93, 95, 97, 98, 99, 100\}.
\end{split}
\end{equation*}
Consider the elliptic curve 
\[
\textcolor{brown}{\EC : y^2 = x^3 + 1.}
\]
Then $\rk_{\Z}\EC(\Q) =0$ but using the SAGE code given below one can check that exactly one of $\EC_1/\Q$ or $\EC_2/\Q$ has rank 0 and the other has positive rank $r$.
We can conclude that the Hilbert's $10^{\text{th}}$ problem is unsolvable for the ring of integers of $\Q(\zeta_3, \sqrt[3]{D})$.

\begin{verbatim}
a = 1
for d in range(1,100): #sometimes the code stopped running; changed range manually
    E=EllipticCurve([0,a])
    E1=EllipticCurve([0,d^2*a])
    E2=EllipticCurve([0,d^4*a])
    print(d , E.rank() , E1.rank(), E2.rank())    
\end{verbatim}
\end{example}

Note: In the following examples, we only explicitly write the set of new $D$'s.

\begin{example}
Fix $D\in \{17, 35, 71\}$.
Consider the elliptic curve 
\[
\textcolor{OliveGreen}{\EC : y^2 = x^3 + 2.}
\]
Then a SAGE/MAGMA check shows that $\rk_{\Z}\EC(\Q) =1$ whereas $\rk_{\Z}\EC_1(\Q) = \rk_{\Z}\EC_2(\Q) = 0$.
\end{example}

\begin{example}
Fix $D$ to be an integer in the set
\[
\{12, 18, 20, 34, 42, 50, 53, 63, 83, 84, 94\}.
\]
Consider the elliptic curve \[\textcolor{Purple}{\EC : y^2 = x^3 + 3.}\]
Then a SAGE/MAGMA check shows that $\rk_{\Z}\EC(\Q) =1$ whereas $\rk_{\Z}\EC_1(\Q) = \rk_{\Z}\EC_2(\Q) = 0$.
\end{example}

\begin{example}
Let $D\in \{15, 38, 43, 57, 60, 62, 74, 75, 79, 85\}$ be an integer.
Consider the elliptic curve 
\[
\textcolor{blue}{\EC : y^2 = x^3 + 4.}
\]
Then $\rk_{\Z}\EC(\Q) =0$ and exactly one of $\EC_1/\Q$ or $\EC_2/\Q$  has rank 0 and the other has positive rank $r$.
\end{example}

\begin{example}
Fix $D\in \{ 5, 6, 30, 36 \}$.
Let 
\[
\textcolor{Orange}{\EC : y^2 = x^3 + 5.}
\]
Then a SAGE check shows that $\rk_{\Z}\EC(\Q) =1$ whereas $\rk_{\Z}\EC_1(\Q) = \rk_{\Z}\EC_2(\Q) = 0$.
\end{example}

\begin{example}
\label{Ex 6}
Let $D\in \{2, 4, 19, 52, 73, 86\}$ be an integer.
Consider the elliptic curve 
\[
\textcolor{purple}{\EC : y^2 = x^3 + 6.}
\]
Then $\rk_{\Z}\EC(\Q) =0$ while exactly one of $\EC_1/\Q$ or $\EC_2/\Q$ has rank 0 and the other has positive rank $r$.
\end{example}

\begin{example}
Let $D\in \{34, 78, 90\}$ be an integer.
Consider the elliptic curve 
\[
\textcolor{Brown}{\EC : y^2 = x^3 + 7.}
\]
Then $\rk_{\Z}\EC(\Q) =0$, exactly one of $\EC_1/\Q$ or $\EC_2/\Q$ has rank 0 and the other has $r>0$.
\end{example}

\begin{example}Let $D\in \{68, 70\}$ be an integer.
Consider the elliptic curve 
\[
\textcolor{MidnightBlue}{\EC : y^2 = x^3 + 13.}
\]
Then $\rk_{\Z}\EC(\Q) =0$ and exactly one of $\EC_1/\Q$ or $\EC_2/\Q$ has rank 0 and the other has positive rank $r$.
\end{example}

\begin{example}
Let $D=22$.
Consider the elliptic curve 
\[
\textcolor{red}{\EC : y^2 = x^3 + 14.}
\]
Then $\rk_{\Z}\EC(\Q) =0$ and exactly one of $\EC_1/\Q$ or $\EC_2/\Q$ has rank 0 and the other has rank $r>0$.
\end{example}

\begin{example}
Let $D=37$.
Consider the elliptic curve 
\[
\textcolor{Rhodamine}{\EC : y^2 = x^3 + 15.}
\]
Then $\rk_{\Z}\EC(\Q) =2$ and we check using SAGE that $\rk_{\Z}\EC_1(\Q) = \rk_{\Z}\EC_2(\Q)=0$.
\end{example}

\begin{example}
Let $D= 58$.
Consider the elliptic curve 
\[
\textcolor{Dandelion}{\EC : y^2 = x^3 + 23.}
\]
Then $\rk_{\Z}\EC(\Q) =0$ and exactly one of $\EC_1/\Q$ or $\EC_2/\Q$ has rank 0 and the other has rank $r>0$.
\end{example}

\begin{example}
Let $D\in \{14, 92\}$ be an integer.
Consider the elliptic curve 
\[
\textcolor{Emerald}{\EC : y^2 = x^3 + 25.}
\]
Then $\rk_{\Z}\EC(\Q) =0$ and exactly one of $\EC_1/\Q$ or $\EC_2/\Q$ has rank 0 and the other has positive rank $r$.
\end{example}

\begin{example}
Let $D=89$.
Consider the elliptic curve 
\[
\textcolor{YellowGreen}{\EC : y^2 = x^3 + 89.}
\]
Then $\rk_{\Z}\EC(\Q) = 2$ and we check that $\rk_{\Z}\EC_1(\Q) = \rk_{\Z}\EC_2(\Q) =0$.
\end{example}

\begin{table}[]
\caption{Values of cube-free $1\le D \le100$ for which unsolvability of Hilbert's 10 Problem is obtained using the method outlined above:}
\begin{center}
\begin{tabular}{ c c c c c c c c c c}
 \textbf{1} & \textcolor{purple}{2} & \textcolor{brown}{3} & \textcolor{purple}{4} & \textcolor{brown}{5} & \textcolor{Orange}{6} & \textcolor{brown}{7} & \textbf{8} & \textcolor{brown}{9} & \textcolor{brown}{10} \\ 
 \textcolor{brown}{11} & \textcolor{Purple}{12} & \textcolor{brown}{13} & \textcolor{Emerald}{14} & \textcolor{blue}{15} & \textbf{16} & \textcolor{OliveGreen}{17} & \textcolor{Purple}{18} & \textcolor{purple}{19} & \textcolor{Purple}{20} \\
 \textcolor{brown}{21} & \textcolor{red}{22} & \textcolor{brown}{23} & \textbf{24} & \textcolor{brown}{25} & \textcolor{brown}{26} & \textbf{27} & \textcolor{brown}{28} & \textcolor{brown}{29} & \textcolor{Orange}{30} \\
 \textcolor{brown}{31} & \textbf{32} & \textcolor{brown}{33} & \textcolor{Brown}{34} & \textcolor{OliveGreen}{35} & \textcolor{Orange}{36} & \textcolor{Rhodamine}{37} & \textcolor{blue}{38} & \textcolor{Purple}{39} & \textbf{40} \\
 \textcolor{brown}{41} & \textcolor{Purple}{42} & \textcolor{blue}{43} & \textcolor{brown}{44} & \textcolor{brown}{45} & \textcolor{brown}{46} & \textcolor{brown}{47} & \textbf{48} & \textcolor{brown}{49} & \textcolor{Purple}{50} \\
 \textcolor{brown}{51} & \textcolor{purple}{52} & \textcolor{Purple}{53} & \textbf{54} & \textcolor{brown}{55} & \textbf{56} & \textcolor{blue}{57} & \textcolor{Dandelion}{58} & \textcolor{brown}{59} & \textcolor{blue}{60} \\
 \textcolor{brown}{61} & \textcolor{blue}{62} & \textcolor{Purple}{63} & \textbf{64} & \textcolor{brown}{65} & \textcolor{brown}{66} & \textcolor{brown}{67} & \textcolor{MidnightBlue}{68} & \textcolor{brown}{69} & \textcolor{MidnightBlue}{70} \\
 \textcolor{OliveGreen}{71} & \textbf{72} & \textcolor{purple}{73} & \textcolor{blue}{74} & \textcolor{brown}{75} & \textcolor{brown}{76} & \textcolor{brown}{77} & \textcolor{Brown}{78} & \textcolor{blue}{79} & \textbf{80} \\ 
 \textbf{81} & \textcolor{brown}{82} & \textcolor{Purple}{83} & \textcolor{Purple}{84} & \textcolor{blue}{85} & \textcolor{purple}{86} & \textcolor{brown}{87} & \textbf{88} & \textcolor{YellowGreen}{89} & \textcolor{Brown}{90} \\ 
 \textcolor{brown}{91} & \textcolor{Emerald}{92} & \textcolor{brown}{93} & \textcolor{Purple}{94} & \textcolor{brown}{95} & \textbf{96} & \textcolor{brown}{97} & \textcolor{brown}{98} & \textcolor{brown}{99} & \textcolor{brown}{100} \\ 
 \end{tabular}
\end{center}
\end{table}


\begin{remark}
    We thank the referee for pointing out to us that using MAGMA one can check that the pairs $(a,D) = (2,71)$ and $(a,D) = (3,83)$ have the desired property.
\end{remark}

\subsection{Follow-up Questions}
\label{Sec: Follow-up questions}
These observations raise the following two obvious questions:

\medskip

\textbf{Question 1:}
For every cube-free integer $D$ does there always exist an elliptic curve of the type
$\EC: y^2 = x^3 + a$
such that on twisting by $D$ exactly one of $\EC, \ \EC_1, \ \EC_2 $ has positive rank (over $\Q$) and the other two have rank 0 (over $\Q$)?

\medskip

Even though this question has implications for Hilbert's $10^{\text{th}}$ Problem, in particular it proves that  Hilbert's $10^{\text{th}}$ Problem is unsolvable in the ring of integers of $\Q(\zeta_3, \sqrt[3]{D})$ for every (cube-free integer) $D$, the question is interesting in its own right.

\medskip

\textbf{Question 2:}
Fix an elliptic curve $\EC: y^2 = x^3 + a$.
For what proportion of cube-free integers $D$, does exactly one of $\EC, \ \EC_1, \ \EC_2$ have positive rank $r$ (over $\Q$) and the other two have rank 0 (over $\Q$)?

\section{Main Result and Proof}
\label{Sec: main result}

\subsection{Hilbert's \texorpdfstring{$10^{\text{th}}$}{} Problem for \texorpdfstring{$\Q(\zeta_3, \sqrt[3]{p})$}{}}

We begin this section by proving some lemmas which will be required for the main theorem.

\begin{lemma}
\label{Lemma 4.1}
Let $p$ be a prime such that $p \equiv 2, 4, 5, 7 \pmod{9}$.
Then Hilbert's $10^{\text{th}}$ problem is unsolvable for the ring of integers of $\Q(\zeta_3, \sqrt[3]{p})$.
\end{lemma}

\begin{proof}
In each of these cases, we prove the unsolvability of Hilbert's $10^{\text{th}}$ Problem in $\Q(\zeta_3, \sqrt[3]{p})$ by applying Lemma \ref{lem1} on an appropriately chosen elliptic curve $E_a$. 
\begin{enumerate}
\item[\textup{(}a\textup{)}]
\textbf{When $p \equiv 2, 5 \pmod{9}$.}
The elliptic curve $\EC_{-432*9^2}$ has rank $1$ over $\Q$. 
As noted in Theorem~\ref{Thm: JMS}(a), both $9p$ and $9p^2$ are not cube-sums i.e.  both the elliptic curves $\EC_{-432*9^2*p^2}$ and $\EC_{-432*9^2*p^4}$ have rank $0$ over $\Q$.
Now the result follows from Lemma \ref{lem1}. 

\item[\textup{(}b\textup{)}]
\textbf{When $p \equiv 4,7 \pmod{9}$.}
Fix an integer $B$  such that $B \notin \mathbb{F}_p^3$.
By Chinese Remainder Theorem, there exists an integer $A$ such that $A \equiv 8 \pmod 9$ and $A \equiv B \pmod p$.
Now by Theorem~\ref{Thm: JMS}(c), there exists a prime $\ell$ such that $ \ell \equiv A \pmod{9p}$ which is a cube-sum and hence $\rk_{\Z} \EC_{-432\ell^2}(\Q) >0$. On the other hand, as mentioned in Theorem~\ref{Thm: JMS}(b)(ii), both $\ell p$ and $\ell p^2$ are not cube-sums and the result follows from Lemma \ref{lem1}. \qedhere

\end{enumerate}
\end{proof}


\begin{lemma}
\label{Lemma 4.2}
For $100\%$ of the primes $p$ (with respect to the natural density) of the form $p \equiv 8 \pmod{9}$, Hilbert's $10^{\text{th}}$ problem is unsolvable for the ring of integers of $\Q(\zeta_3, \sqrt[3]{p})$.
\end{lemma}

\begin{proof}
Consider the infinite set consisting of primes
$S= \{ \ell \mid \ell \equiv 4, 7 \pmod 9\textrm{ and } 3 \notin \mathbb{F}^3_\ell\} $. Enumerate the primes of $S$ as $\ell_1 =7$, $\ell_2 = 13$, $\ell_3=31, \ldots$ Further note that for every $\ell\in S$, the prime $\ell$ is a cube-sum by Theorem~\ref{Thm: JMS}(b)(i). 

We have $\ell_1=7$, now take a prime $p \equiv 8 \pmod{9}$ such that $p \not\equiv \pm 1 \pmod{\ell_1}$, i.e., choose a prime $p \equiv 8 \pmod{9}$ such that $p \notin \mathbb{F}^3_{\ell_1}$.
By Theorem~\ref{Thm: JMS}(b)(ii), both $\ell_1 p$ and $\ell_1 p^2$ are not cube-sum.
By Lemma \ref{lem1}, Hilbert's $10^{\text{th}}$ problem is unsolvable for the ring of integers of $L = \Q(\zeta_3, \sqrt[3]{p})$.
Thus among  all the primes $p$ of the form $p\equiv 8\pmod{9}$, for two-thirds of them (corresponds to those which are not of the form $\pm 1 \pmod 7$) Hilbert's $10^{\text{th}}$ problem is unsolvable for $\mathcal{O}_L$.

Next we work with the prime $\ell_2 =13$. 
We choose a prime $p$ such that $p \equiv 8 \pmod9$, $p \equiv \pm 1 \pmod 7$, and $p \not\equiv \pm 1, \pm 8 \pmod{13}$, that is $p \equiv 8 \pmod 9$, $p \in \mathbb{F}^3_{\ell_1}$ and $p \notin \mathbb{F}^3_{\ell_2}$.
Again by Theorem~\ref{Thm: JMS}(b)(ii) both $\ell_2 p$ and $\ell_2 p^2$ are not cube-sum and hence, we can conclude by Lemma \ref{lem1} Hilbert's $10^{\text{th}}$ problem is unsolvable for $\mathcal{O}_L$.
Consequently, we can now conclude that among all the of primes $p$ of the form $p \equiv 8 \pmod{9}$ and $p \equiv \pm 1 \pmod 7$, for $2/3$-rd of them Hilbert's $10^{\text{th}}$ problem is unsolvable for $\mathcal{O}_L$. 

Combining the two cases,  for $\frac{2}{3} + \left(\frac{1}{3} \times \frac{2}{3} \right) = \frac{8}{9}$-th of the  primes $p$ of the form $p \equiv 8 \pmod{9}$, Hilbert's $10^{\text{th}}$ problem is unsolvable for $\mathcal{O}_L$.

Proceeding in this way, we see that at the $k$-step, for a prime $p$ with $p \equiv 8 \pmod 9$, $p \in \mathbb{F}^3_{\ell_i}$ for $1 \le i \le k$ and $p \notin \mathbb{F}^3_{\ell_{k+1}}$, Hilbert's $10^{\text{th}}$ problem is unsolvable for $\mathcal{O}_L$.
Consequently, we conclude that for $(1- \frac{1}{3^{k+1}})$ of primes $p$ of the form $p \equiv 8 \pmod{9}$, Hilbert's $10^{\text{th}}$ problem is unsolvable for $\mathcal{O}_L$.
The set $S$ is infinite; by induction it follows that for $100\%$ of primes $p$ of the form $p \equiv 8 \pmod{9}$, Hilbert's $10^{\text{th}}$ problem is unsolvable for $\mathcal{O}_L$. 
\end{proof}

\begin{theorem}
For $\frac{5}{6}$-th (with respect to the natural density) of all primes $p$, Hilbert's $10^{\text{th}}$ problem is unsolvable for the ring of integers of $\Q(\zeta_3, \sqrt[3]{p})$.
\end{theorem}

\begin{proof}
The result follows immediately from Lemmas~\ref{Lemma 4.1} and \ref{Lemma 4.2}.
\end{proof}

\subsection{Hilbert's \texorpdfstring{$10^{\text{th}}$}{} Problem for \texorpdfstring{$\Q(\zeta_3, \sqrt{D}, \sqrt[3]{p})$}{}}

Once we have established that Hilbert's $10^{\text{th}}$ Problem is unsolvable in the degree 6 extension $\Q(\zeta_3, \sqrt[3]{p})$, we can extend it to certain degree 12 extensions by using the corollary of Shlapentokh's theorem.

These are new examples of rings of integers of number fields where Hilbert's $10^{\text{th}}$ problem is being shown to be unsolvable.
For the argument in \cite{kundu2024studying}, the authors found $\EC/\Q$ with rank 0 over $\Q(\zeta_3, \sqrt[3]{p})$ and positive rank over $\Q(\sqrt{D})$; this \textit{did not} provide examples of degree 12 number fields $L_{D,p} = \Q(\zeta_3, \sqrt{D}, \sqrt[3]{p})$ where Hilbert's $10^{\text{th}}$ problem is unsolvable for $\mathcal{O}_{L_{D,p}}$ because the unsolvability of Hilbert's $10^{\text{th}}$ problem for the ring of integers of $\Q(\zeta_3, \sqrt[3]{p})$ was not verified. 

\begin{proposition}\label{Lemma 4.4}
Let $S$ be the set of square-free integers such that for every $D \in S$, the rank of the elliptic curve $\EC_{-432D^3}$ is positive and $T$ be the set of primes $p$ of the form $p \equiv 2,5 \pmod{9}$. 
Hilbert's $10^\text{th}$ problem is unsolvable for the ring of integers of $L_{D,p}$ for every $D \in S$ and every $p \in T$.

Further, the set $S$ is infinite and moreover for $X \gg 0$, we have 
\[
\#(S \cap [-X,X]) \gg X^{1 - \epsilon} \textrm{ for any  } \epsilon >0.
\]
\end{proposition}

\begin{proof}
Consider the elliptic curve $\EC_{-432}$.
By our assumption on $S$, we know that 
\[
\rk_{\Z} \EC_{-432}(\Q(\sqrt{D})) = \rk_{\Z} \EC_{-432}(\Q) + \rk_{\Z} \EC_{-432D^3}(\Q) >0.
\]
Recall that for $p \in T$, by Theorem~\ref{Thm: JMS}(a), both $p$ and $p^2$ are not cube-sums, i.e. $\rk_{\Z} \EC_{-432p^2}(\Q)= \rk_{\Z} \EC_{-432p^4}(\Q)$. Further recall that $\rk_{\Z} \EC_{k}(\Q(\zeta_3)) = 2 \rk_{\Z} \EC_{k}(\Q)$.
Then it follows from \eqref{eq1} that 
\[
\rk_{\Z} \EC_{-432}(\Q(\zeta_3, \sqrt[3]{p})) = \rk_{\Z} \EC_{-432}(\Q(\zeta_3)) + \rk_{\Z} \EC_{-432p^2}(\Q(\zeta_3)) + \rk_{\Z} \EC_{-432p^4}(\Q(\zeta_3)) =0.
\]
In Lemma~\ref{Lemma 4.1}, we have shown that Hilbert's $10^\text{th}$ problem is unsolvable over $\Q(\zeta_3, \sqrt[3]{p})$.
Hence using Theorem~\ref{GFP prop 3.3}, we can guarantee that Hilbert's $10^\text{th}$ problem is unsolvable over $\mathcal{O}_{L_{D,p}}$.

To prove the second assertion, we argue as follows: let $\EC$ be an elliptic curve over $\Q$ and $D$ be a square-free integer co-prime to the conductor of $\EC$.
Let $\EC^{(D)}$ denote the quadratic twist of $\EC$ by $D$ and let $L(\EC,s)$ represent  the complex $L$-function of $\EC/\Q$.
Put 
\[
N_1(\EC,X):= \{ |D|<X \mid \operatorname{ord}_{s=1}\  L(\EC^{(D)}, s) = 1 \}.
\]
Then by \cite{PP}, for $X \gg 0$, we have $N_1(\EC,X) \gg_{\epsilon} X^{1 - \epsilon}$, for any $\epsilon >0$. 
By the work of B.~Gross--D.~Zagier and V.~Kolyvagin,  it is known  that for an elliptic curve $\mathsf{A}/\Q$,  $\rk_\Z \mathsf{A}(\Q) =1$ whenever $\operatorname{ord}_{s=1} \ L(\mathsf{A}, s) = 1 $.
By applying this argument to $\EC=\EC_{-432}$ and $\mathsf{A} = \EC^{(D)}$, the assertion follows.
\end{proof}


\begin{remark}
Note that for the square-free integers of the form (i) $ \pm (3k+2)$  with $k>0$ and (ii) $ \pm 3(3k+1)$ with $k>0$, the root number computations due to B.~Birch--N.~Stephens implies that the global root number of $\EC_{-432D^3}/\Q$, denoted by $\omega(\EC_{-432D^3})$ is $-1$.
Thus if we make an additional hypothesis that $\dim_{\mathbb{F}_3} \Sha(\EC_{-432D^3}/\Q)[3]$ is even (see \cite[\S 5]{jha20223}) for all square-free integers $D$, then the natural density of $S$ is at least $\frac{1}{2}$.
Note that for an elliptic curve $\EC$ over a number field $K$, the hypothesis $\dim_{\mathbb{F}_2} \Sha(\EC/K)[2]$ being even  is discussed  in \cite{mazur2010ranks}.
    
\end{remark}

\bibliographystyle{amsalpha}
\bibliography{references}

\end{document}